\documentclass[10pt,a4paper,reqno]{amsart}
\usepackage{amssymb}
\usepackage{amsmath}
\usepackage{enumerate}
\usepackage{graphicx}
\usepackage{tjmm}
\usepackage{multirow}

\usepackage[utf8x]{inputenc}

\begin{document}

\title[Randomized orthogonal decomposition]{Mathematical considerations on Randomized Orthogonal Decomposition method for developing twin data models}
\author{Diana A. Bistrian}
\address{
 University Politehnica of Timisoara\newline
\indent Department of Electrical Engineering and Industrial Informatics\newline
\indent Revolutiei Nr.5, 331128 Hunedoara, Romania}
\email{diana.bistrian@upt.ro}

\setcounter{page}{0}
\coordinates{14}{2022}{2}{105-115}


\begin{abstract}

This paper introduces the approach of Randomized Orthogonal Decomposition (ROD) for producing twin data models in order to overcome the drawbacks of existing reduced order modelling techniques. When compared to Fourier empirical decomposition, ROD provides orthonormal shape modes that maximize their projection on the data space, which is a significant benefit.
A shock wave event described by the viscous Burgers equation model is used to illustrate and evaluate the novel method.
The new twin data model is thoroughly evaluated using certain criteria of numerical accuracy and computational performance.

\end{abstract}

\maketitle

\section{Introduction}

The discovery of a realistic approximation of the complicated response of raw data by models of low complexity, i.e. reduced order models (ROM), has drawn the attention of researchers in recent years. This is conceivable given that these complex systems are dominated by a variety of underlying patterns with varying contributions in the reconstitution of the data \cite{Holmes1996}. ROM models have different computational properties based on the mathematical strategy used to create them.
A twin data model is a substitute model whose primary function is to replicate the behavior of the original process. The key benefit of coupling the dynamical process with a simplified twin data model is to accurately map the dynamics to timeframes where it suffers from considerable changes and is therefore challenging to examine.

 The most well-known methods for creating ROM models are currently proper orthogonal decomposition (POD), which is based on Fourier empirical decomposition and dynamic mode decomposition (DMD). The principles and disadvantages of these two methods are briefly presented below.

 For purposes of reduced order modelling, numerous practitioners have embraced the POD technique, for exemplification see
 \cite{Kaiser2018, XiaoBis2019, Stefanescu2019, XiaoFang2019, LiKaiser2019, WangFangxin2021, BruntonBook2022}.
 The intrusive model order reduction is usually derived by combining POD with Galerkin projection methods \cite{Bistrian2015, San2014}.
 Because the Galerkin projection is theoretically conducted by arduous computation and requires stabilizing procedures in the course of numerical implementation, this methodology has problems with efficiency and lacks stability \cite{ChKutz2019, MauroyMezicBook2019, Ahmed2020, Iliescu2022}.

  Understanding and refining the DMD technique has received a lot of attention, and various DMD variations have been made available, see \cite{chenoptimal2012, Tu2014, Kutz2016, Noack2016, Erichson2019, Bistrian2017, AhmedSanBistrian2022}.
 Practitioners of modal decomposition frequently debate the choice of DMD modes to be employed for the flow reconstruction \cite{Noack2011, Tissot2014}. The offline part of the procedure takes extra care and the CPU time is increased by specifying a DMD modes'selection criterion \cite{Bistrian2017a}.
 Another disadvantage of the DMD method  is that it doesn't create orthogonal modes, necessitating a significantly large number of modes.

  To overcome the aforementioned problems with the existing approaches, this research presents a novel technique for developing fluid dynamics twin data models. This paper introduces the algorithm of Randomized Orthogonal Decomposition (ROD) for creating twin data models having by definition the smallest error and the highest correlation in relation to the original data.
It will be shown that a key advantage of ROD is the availability of orthonormal shape modes that maximize their projection on the data space in comparison with Fourier empirical decomposition.

The viscous Burgers equation model's description of a shock wave event is used to demonstrate and analyze the new algorithm's performance. The numerical results are presented for a detailed evaluation of the new twin data model utilizing particular criteria such as numerical accuracy and computational efficiency.

The remainder of the article is organized as follows.
Section 2 discusses the mathematical aspects of the randomized orthogonal decomposition approach.
Section 3 presents the test example with the exact mathematical solution.
Section 4 presents ROD's twin data model and conducts a qualitative study of the model.
A summary and conclusions are given in Section 5.

\section{Randomized Orthogonal Decomposition}

Suppose that $D = \left[ {0,L} \right] \subset \mathbb{R}$ represents the computational domain and let the Hilbert space ${L^2}\left( D \right)$  of square integrable functions on $D$
\begin{equation}
{L^2}\left( D \right) = \left\{ {\phi :D \to \mathbb{R}\left| {\int_D {{{\left| \phi  \right|}^2}dx < \infty } } \right.} \right\}
\end{equation}
to be endowed with the inner product
\begin{equation}
{\left\langle {{\phi _i}\left( x \right),{\phi _j}\left( x \right)} \right\rangle _{{L^2}\left( D \right)}} = \int_D {{\phi _i}\left( x \right)} \,{\phi _j}\left( x \right)\,dx\quad for\;{\phi _i},{\phi _j} \in {L^2}\left( D \right)
\end{equation}
and the induced norm ${\left\| \phi  \right\|_{{L^2}\left( D \right)}} = \sqrt {{{\left\langle {\phi ,\phi } \right\rangle }_{{L^2}(D)}}} $  for $\phi  \in {L^2}\left( D \right)$.

The data ${u_i}\left( {x,t} \right) = u\left( {x,{t_i},} \right),\;{t_i} = i\Delta t,\;i = 0,...,{N_t}$, represent measurements at the constant sampling time $\Delta t$, $x$ representing the Cartesian spatial coordinate.

The data matrix whose columns represent the individual data samples is called \textit{the snapshot matrix}
\begin{equation}
V = \left[ {\begin{array}{*{20}{c}}
{{u_0}}&{{u_1}}&{...}&{{u_{N_t}}}
\end{array}} \right] \in {\mathbb{R}^{{N_x} \times ({N_t} + 1)}}.
\end{equation}
Each column ${u_i} \in {\mathbb{R}^{N_x}}$ is a vector with ${N_x}$ components, representing the spatial measurements corresponding to the ${N_t} + 1$ time instances.

\begin{proposition} \textbf{(Fourier Empirical Orthogonal Decomposition)}
Let
\[V = \left[ {\begin{array}{*{20}{c}}
{{u_0}}&{{u_1}}&{...}&{{u_{N_t}}}
\end{array}} \right] \in {\mathbb{R}^{{N_x} \times ({N_t} + 1)}}\]
be a real-valued data matrix of rank $r \le \min \left( {{N_x},{N_t} + 1} \right)$, whose columns ${u_j} \in {\mathbb{R}^{{N_x}}}$, $j = 1,...,{N_t} + 1$ are data snapshots.
The Singular Value Decomposition (SVD) yields the factorization
\begin{equation}\label{f1}
V = \Psi \Sigma {\Phi ^T}
\end{equation}
%
of the matrix $V$, where $\Psi  = \left[ {{\psi _1},...,{\psi _{{N_x}}}} \right] \in {\mathbb{R}^{{N_x} \times {N_x}}}$ and
$\Phi  = \left[ {{\varphi _1},...,{\varphi _{{{N_t} + 1}}}} \right] \in {\mathbb{R}^{({N_t} + 1) \times ({N_t} + 1)}}$ are orthogonal matrices,
\[\Sigma  = \left( {\begin{array}{*{20}{c}}
D&0\\
0&0
\end{array}} \right) \in {\mathbb{R}^{{N_x}  \times ({N_t} + 1)}},\]
with $D = diag\left( {{\sigma _1},...,{\sigma _r}} \right) \in {\mathbb{R}^{r \times r}}$
and ${\sigma _1} \ge {\sigma _2} \ge ... \ge {\sigma _r} > 0.$
Then $\left\{ {{\psi _i}} \right\}_{i = 1}^r$ and $\left\{ {{\varphi _i}} \right\}_{i = 1}^r$ represent the eigenvectors of
$V{V^T}$ and ${V^T}V$, respectively, with eigenvalues ${\lambda _i} = \sigma _i^2 > 0$ for $i = 1,...,r$, due to the relations
\[V{\varphi _i} = {\sigma _i}{\psi _i}\quad and\quad {V^T}{\psi _i} = {\sigma _i}{\varphi _i}\quad for\;i = 1,...,r.\]

It follows that
\begin{equation}\label{f3}
\begin{array}{l}
{u_j} = {\sum\limits_{i = 1}^r {\left[ {D{\Phi ^T}} \right]} _{ij}}{\psi _i} = {\sum\limits_{i = 1}^r {\left[ {\overbrace {{\Psi ^T}\Psi }^{ = I \in {\mathbb{R}^{r \times r}}}D{\Phi ^T}} \right]} _{ij}}{\psi _i}\\
\quad  = {\sum\limits_{i = 1}^r {\left[ {{\Psi ^T}V} \right]} _{ij}}{\psi _i} = \sum\limits_{i = 1}^r {{{\left\langle {{u_j},{\psi _i}} \right\rangle }_{{L^2}(D)}}} {\psi _i}.
\end{array}
\end{equation}

In the case of linearly-independent snapshots,
the empirical orthogonal decomposition involves a number of terms equal to $r = \min \left( {{N_x},{N_t} + 1} \right)$, which can be a large number.
By convention, let suppose that  ${\rm{r = min}}\left( {{N_x},{N_t} + 1} \right) = {N_x}$. The following model
\begin{equation}\label{fullF}
u^{Fourier}\left( {x,t} \right) = \sum\limits_{i = 1}^{{N_x}} {{a_i}\left( t \right){\psi _i}\left( x \right)} ,\quad {a_i}\left( t \right) = {\left\langle {u,{\psi _i}} \right\rangle _{{L^2}\left( D \right)}}
\end{equation}
is called \textit{Fourier empirical orthogonal decomposition}.

\end{proposition}

It is more convenient to seek an orthonormal base of functions
\begin{equation}
 \Phi  = \left\{ {{\phi _1},{\phi _2},...} \right\},\quad {\left\langle {{\phi _i}\left( x \right),{\phi _j}\left( x \right)} \right\rangle _{{L^2}\left( D \right)}} = {\delta _{ij}},\quad {\left\| \phi  \right\|_{{L^2}\left( D \right)}} = 1,
\end{equation}
where ${\delta _{ij}}$ is the Kronecker delta symbol, consisting of a minimum number of functions ${{\phi _i}\left( x \right)}$, such that the approximation of $u\left( {x,t} \right)$ through this base is as good as possible, in order to create a twin data model of reduced complexity.

The twin data model (DTM) at every time step $\left\{{{{\rm{t}}_1}{\rm{,}}...{\rm{,}}{{\rm{t}}_{N_t}}} \right\}$ is written according to the following relation:
\begin{equation}\label{dtm}
{u^{DTM}}\left( {x,{t_i}} \right) = \sum\limits_{j = 1}^{{N_{DTM}}} {\underbrace {{a_j}\left( {{t_i}} \right)}_{Modal\;amplitudes}\overbrace {{\phi _j}\left( x \right)}^{Leading\;shape\;modes},} \;\;{t_i} \in \left\{ {{{\rm{t}}_1},...,{{\rm{t}}_{{N_t}}}} \right\},
\end{equation}
where ${\phi _j} \in \mathbb{C}$ represent the $\Phi $ base functions, which we call \textit{the leading shape modes}, ${N_{DTM}} \ll \min \left( {{N_x},{N_t} + 1} \right)$ represents the number of terms in the representation (\ref{dtm}) which we impose to be minimal and ${a_j}\left( {{t_i}} \right) $ represent the modal growing amplitudes.

Determination of the optimal decomposition (\ref{dtm}) then amounts to finding the solution to the following multiobjective constrained optimization problem:
\begin{equation}\label{optimprob}
\begin{array}{*{20}{l}}
{\mathop {\min }\limits_{{\phi _j},{a_j},{N_{DTM}}} \sum\limits_{i = 1}^{{N_x}} {} \int\limits_D {\left\| {u\left( {x,{t_i}} \right) - \sum\limits_{j = 1}^{{N_{DTM}}} {{a_j}\left( {{t_i}} \right){\phi _j}\left( x \right)} } \right\|_{{L^2}\left( D \right)}^2dx} ,}\\
{\mathop {\min }\limits_{{\phi _j},{a_j},{N_{DTM}}} \sum\limits_{i = 1}^{{N_x}} {} \int\limits_D {\frac{{ - \left\| {u\left( {x,{t_i}} \right)\sum\limits_{j = 1}^{{N_{DTM}}} {{a_j}\left( {{t_i}} \right){\phi _j}\left( x \right)} } \right\|_{{L^2}\left( D \right)}^2}}{{\left\| {u{{\left( {x,{t_i}} \right)}^H}u\left( {x,{t_i}} \right)} \right\|_{{L^2}\left( D \right)}^2\left\| {{{\left( {\sum\limits_{j = 1}^{{N_{DTM}}} {{a_j}\left( {{t_i}} \right){\phi _j}\left( x \right)} } \right)}^H}\sum\limits_{j = 1}^{{N_{DTM}}} {{a_j}\left( {{t_i}} \right){\phi _j}\left( x \right)} } \right\|_{{L^2}\left( D \right)}^2}}} dx,}\\
{s.t.\quad {{\left\langle {{\phi _i},{\phi _j}} \right\rangle }_{{L^2}\left( D \right)}} = {\delta _{ij}},\quad {{\left\| \phi_{i}  \right\|}_{{L^2}\left( D \right)}} = 1,\quad 1 \le i \le j \le {N_{DTM}}}
\end{array}
\end{equation}
where $H$ denotes the conjugate transpose of the snapshot  containing the data.

\begin{definition} \textbf{(The Projection Operator)}

Let
\[{V_0} = \left[ {\begin{array}{*{20}{c}}
{{u_0}}&{{u_1}}&{...}&{{u_{{N_t} - 1}}}
\end{array}} \right] \equiv \left\{ {u_i^0} \right\}_{i = 0}^{{N_t} - 1} \in {\mathbb{R}^{{N_x} \times {N_t}}}\]
be a real-valued data matrix, whose columns are data snapshots.

Let
\[\Phi  = \left\{ {{\phi _1},{\phi _2},...,{\phi _{{N_{DTM}}}}} \right\}\]
be the shape modes base.

We define the bounded projection operator ${P_{{V_0}}}\Phi$ , that maps every shape mode $\left\{ {{\phi _i}} \right\}_{i = 1}^{{N_{DTM}}}$ onto its projection on the data vectors
$\left\{ {u_i^0} \right\}_{i = 0}^{{N_t} - 1}$ along the computational domain direction $D$:

\begin{equation}\label{proj}
{P_{{V_0}}}\Phi \left( {{\phi _i},u_j^0} \right) \equiv {P_{u_j^0}}{\phi _i} = \left( {{{\left\langle {{\phi _i},u_j^0} \right\rangle }_{{L^2}\left( D \right)}}/{{\left\langle {u_j^0,u_j^0} \right\rangle }_{{L^2}\left( D \right)}}} \right)u_j^0.
\end{equation}

\end{definition}

\begin{proposition} \textbf{(Randomized Singular Value Decomposition of rank $k$)}\label{Randsvd}

Let
\[{{V_{0}} = \left[ {\begin{array}{*{20}{c}}
{{u_0}}&{{u_1}}&{...}&{{u_{{N_t} - 1}}}
\end{array}} \right] \in {\mathbb{R}^{{N_x} \times {N_t}}}}\]
be a real-valued data matrix, whose columns are data snapshots.
If we impose a target rank ${k < \min \left( {{N_x},{N_t}} \right)}$, the Randomized Singular Value Decomposition of rank $k$ ($k$-RSVD) produces $k$ left singular vectors of $V_{0}$ and  has the following steps:
\begin{enumerate}

\item[1.] Generate a Gaussian random test matrix $M$ of size ${N_t} \times k$.

\item[2.] Compute a compressed sampling matrix by multiplication of data matrix with random matrix $Q = V_{0}M$.

\item[3.] Project the data matrix to the smaller space $P = {Q^H}V_{0}$, where $H$ denotes the conjugate transpose.

\item[4.] Produce the economy-size singular value decomposition of low-dimensional data matrix $\left[ {T,\Sigma ,W} \right] = svd\left( P \right)$.

\item[5.] Compute the right singular vectors $U = QT$,
 $U \in {\mathbb{R}^{{N_x} \times k}}$, $\Sigma  \in {\mathbb{R}^{k \times k}}$, $W \in {\mathbb{R}^{{N_t}
\times k}}$.

\end{enumerate}

\end{proposition}

\begin{theorem} \textbf{(Randomized Orthogonal Decomposition: ROD)}\label{Th1}

Let
\begin{equation}
\begin{array}{l}
{V_0} = \left[ {\begin{array}{*{20}{c}}
{{u_0}}&{{u_1}}&{...}&{{u_{{N_t} - 1}}}
\end{array}} \right] \equiv \left\{ {u_i^0} \right\}_{i = 0}^{{N_t} - 1} \in {\mathbb{R}^{{N_x} \times {N_t}}},\;\\
{V_1} = \left[ {\begin{array}{*{20}{c}}
{{u_1}}&{{u_2}}&{...}&{{u_{{N_t}}}}
\end{array}} \right] \equiv \left\{ {u_i^1} \right\}_{i = 1}^{{N_t}} \in {\mathbb{R}^{{N_x} \times {N_t}}},
\end{array}
\end{equation}
two time-shifted data matrices with rank ${\rm{r }} \le {\rm{ min}}\left( {{N_x},{N_t}} \right)$, whose columns are data snapshots.
Then for $1 \le {N_{DTM}} \le r$, the optimization problem
\begin{equation}\label{optmodes}
\begin{array}{*{20}{l}}
{\quad \quad \quad \mathop {\max }\limits_{{\phi _1},...,{\phi _{{N_{DTM}}}}} \sum\limits_{i = 1}^{{N_{DTM}}} {\sum\limits_{j = 1}^{{N_t}} {\left\| {{P_{u_j^0}}{\phi _i}} \right\|_{{L^2}\left( D \right)}^2} } \quad }\\
{s.t.\quad {{\left\langle {{\phi _i},{\phi _j}} \right\rangle }_{{L^2}\left( D \right)}} = {\delta _{ij}},\quad {{\left\| {{\phi _i}} \right\|}_{{L^2}\left( D \right)}} = 1,\quad 1 \le i \le j \le {N_{DTM}}}
\end{array}
\end{equation}
is being solved by the subspace $span\left\{ {{\phi _1},{\phi _2},...,{\phi _{{N_{DTM}}}}} \right\}$ spanned by the sequence of orthonormal functions
\begin{equation}\label{modesdef}
\left\{ {{\phi _i}} \right\}_{i = 1}^{{N_{DTM}}} = {\left\langle {U,{X_{.,i}}} \right\rangle _{{L^2}\left( D \right)}}/{\left\| {\left\langle {U,{X_{.,i}}} \right\rangle } \right\|_{{L^2}\left( D \right)}}
\end{equation}
where $U$ represents the matrix of left singular vectors produced by Randomized Singular Value Decomposition of rank $N_{DTM}$ of data matrix $V_{0}$ and $X$ denotes the eigenvectors to the Koopman propagator operator $\mathcal A$, i.e.
${u_{{N_t}}} = {{\mathcal A}^{{N_t}}}{u_0}$.

%
\end{theorem}

\begin{proof}

Following the Koopman decomposition assumption \cite{Koopm1931},  we consider that a propagator operator $\mathcal{A}$ exists, that maps every column vector onto the next one, i.e.
\begin{equation}\label{steps1}
\left\{ {{u_0},\;{u_1} = {\mathcal A}{u_0},\;{u_2} = {\mathcal A}{u_1} = {{\mathcal A}^2}{u_0},.\;..,\;{u_{N_t}} = {\mathcal A}{u_{{N_t} - 1}} = {{\mathcal A}^{{N_t}}}{u_0}} \right\}.
\end{equation}

For a sufficiently long sequence of the snapshots, suppose that the last snapshot $u_{N_t}$ can be written as a linear combination of previous
${N_t}$ vectors, such that
\begin{equation}\label{linearcomb}
u_{N_t} = {c_0}{u_0} + {c_1}{u_1} + ... + {c_{{N_t} - 1}}{u_{{N_t} - 1}} + \mathcal{R},
\end{equation}
in which ${{\rm{c}}_i} \in {\rm{\mathbb{R},i  =  0,}}...{\rm{,{N_t}  -  1}}$  and $\mathcal{R}$ is the residual vector.  The following relations are true:
\begin{equation}\label{kryl}
\left\{ {{u_1},{u_2},...{u_{N_t}}} \right\} = {\mathcal A}\left\{ {{u_0},{u_1},...{u_{{N_t} - 1}}} \right\} = \left\{ {{u_1},{u_2},...,V_0c} \right\} + {\mathcal R},
\end{equation}
where $c = {\left( {\begin{array}{*{20}{c}} {{c_0}}&{{c_1}}&{...}&{{c_{{N_t} - 1}}}\end{array}} \right)^T}$ is the unknown column vector.

Thus, the aim is to solve the following eigenvalue problem:
\begin{equation}\label{approx}
V_1 = \mathcal{A}V_0 = V_0\mathcal{S} + \mathcal{R},
\end{equation}
where $\mathcal{S}$ approximates the eigenvalues of $\mathcal{A}$ when ${\left\| \mathcal{R} \right\|_2} \to 0$. This is equivalent to solve the minimization problem:
\begin{equation}\label{minpro}
\mathop {min}\limits_\mathcal{S} \;\mathcal{R} = {\left\| {V_1 - V_0\mathcal{S}} \right\|_2},
\end{equation}
where ${\left\| {\, \cdot \,} \right\|_2}$ is the ${L_2}$-norm of ${\mathbb{R}^{{N_x}}}$.

The solution to the minimization problem (\ref{minpro}) is found in the following manner.
Suppose that ${\rm{r }} \le {\rm{ min}}\left( {{N_x},{N_t}} \right)$.
Then for $1 \le {N_{DTM}} \le r$,
we identify the $N_{DTM}$-RSVD of $V_{0}$, that yields the factorization:
\begin{equation}
V_{0}= U\Sigma {W^H},
\end{equation}
of the snapshot matrix $V_{0}$,
where $U = \left[ {{u_1},...,{u_{{N_{DTM}}}}} \right] \in {\mathbb{R}^{{N_x} \times {N_{DTM}}}}$ and $W = \left[ {{w_1},...,{w_{{N_{DTM}}}}} \right] \in {\mathbb{R}^{{N_t} \times {N_{DTM}}}}$ are orthogonal matrices that
contain the eigenvectors of ${{\rm{V}}_0}{{\rm{V}}_0}^H$ and ${{\rm{V}}_0}^H{{\rm{V}}_0}$, respectively,
 $\Sigma  = diag\left( {{\sigma _1},...,{\sigma _{N_{DTM}}}} \right) \in {\mathbb{R}^{N_{DTM} \times N_{DTM}}}$
is a square diagonal matrix containing the singular values of $V_{0}$ and ${H}$ means the conjugate transpose.

Relations $\mathcal{A}V_0= V_1= V_0S + \mathcal{R}, {\left\| \mathcal{R} \right\|_2}
\to 0$ and $V_0 = U\Sigma {W^H}$ yield:
\[\mathcal{A}U\Sigma {W^H} = V_1 = U\Sigma {W^H}\mathcal{S}\quad \Rightarrow \quad {U^H}\mathcal{A}U\Sigma {W^H} = {U^H}U\Sigma {W^H}\mathcal{S}\quad \Rightarrow \quad \;\mathcal{S} = {U^H}\mathcal{A}U.     \]

     From $\mathcal{A}U\Sigma {W^H} = V_1$ it follows that $\mathcal{A}U = V_1W{\Sigma ^{ - 1}}$ and hence
$\mathcal{S} = {U^H}\left( {V_1W{\Sigma ^{ - 1}}}
\right).$

As a consequence, the solution to the minimization problem (\ref{minpro}) is the matrix operator
\begin{equation}\label{defS}
  \mathcal{S} = {U^H}\left( {V_1W{\Sigma ^{ - 1}}} \right).
\end{equation}

The eigenvalues and eigenvectors of $\mathcal{S}$ will converge toward the eigenvalues and eigenvectors of the Koopman propagator operator $\mathcal{A}$ as a direct result of solving the minimization problem (\ref{minpro}), which improves overall convergence.

Let $X \in {\mathbb{R}^{{N_{DTM}} \times {N_{DTM}}}}$,
$\Lambda  \in {\mathbb{R}^{{N_{DTM}} \times {N_{DTM}}}}$
be the eigenvectors, respectively the eigenvalues of the data propagator matrix $\mathcal{S}$:
\begin{equation}\label{Seig}
  \mathcal{S}X=X\Lambda.
\end{equation}

Let $\Phi=span\left\{ {{\phi _1},{\phi _2},...,{\phi _{{N_{DTM}}}}} \right\}$ be the subspace spanned by the sequence of functions
\begin{equation}\label{md1}
\left\{ {{\phi _i}} \right\}_{i = 1}^{{N_{DTM}}} = \left\langle {U,{X_{.,i}}} \right\rangle _{{L^2}\left( D \right)}
\end{equation}
where $U$ represents the matrix of left singular vectors produced by Randomized Singular Value Decomposition of rank $N_{DTM}$ of data matrix $V_{0}$. It follows that
\begin{equation}
{\left\langle {{\phi _i},{\phi _j}} \right\rangle }_{ {L^2}\left( D \right)} = {\delta _{ij}},\quad 1 \le i \le j \le {N_{DTM}}
\end{equation}
i.e., $\Phi$ forms an orthogonal base to the data space. The $\Phi$ base vectors maximize their projection on the data space and they represent the solution to the constrained optimization problem (\ref{optimprob}), therefore they produce the twin data model with the expression given by Eq.(\ref{dtm}).

\end{proof} 

\begin{corollary}
The base vectors $\left\{ {{\phi _i}} \right\}_{i = 1}^{{N_{DTM}}}$ defined by Eq.(\ref{modesdef}), respectively their corresponding modal coefficients $\left\{ {{a_i}} \right\}_{i = 1}^{{N_{DTM}}} = \left\langle {U,{X_{.,i}}} \right\rangle _{{L^2}\left( D \right)}$, solve the multiobjective constrained minimization problem (\ref{optimprob}).

\end{corollary}

\begin{corollary}
If
\begin{equation}\label{projDTM}
\left\| {{P_{{V_0}}}\Phi } \right\|_{{L^2}\left( D \right)}^2 = \frac{1}{{{N_{DTM}}}}\sum\limits_{i = 1}^{{N_{DTM}}} {\sum\limits_{j = 1}^{{N_t}} {\left\| {{P_{u_j^0}}{\phi _i}} \right\|_{{L^2}\left( D \right)}^2} }
\end{equation}
represents the mean squared sum of the norms of projections of shape modes $\left\{ {{\phi _i}} \right\}_{i = 1}^{{N_{DTM}}}$ produced by Randomized Orthogonal Decomposition
on the data space $V_0$, and
\begin{equation}\label{projF}
  \left\| {{P_{{V_0}}}\Psi } \right\|_{{L^2}\left( D \right)}^2 = \frac{1}{{{N_x}}}\sum\limits_{i = 1}^{{N_x}} {\sum\limits_{j = 1}^{{N_t}} {\left\| {{P_{u_j^0}}{\psi _i}} \right\|_{{L^2}\left( D \right)}^2} }
\end{equation}
represents the mean squared sum of the norms of projections in the case of the Fourier empirical orthogonal modes  $\left\{ {{\psi _i}} \right\}_{i = 1}^{{N_x}}$ on the same data space, then
\begin{equation}\label{compproj}
\left\| {{P_{{V_0}}}\Phi } \right\|_{{L^2}\left( D \right)}^2 > \left\| {{P_{{V_0}}}\Psi } \right\|_{{L^2}\left( D \right)}^2.
\end{equation}
\end{corollary}

This means that, for the purpose of twin modelling, the ROD shape modes defined by Eq.(\ref{modesdef}) are qualitatively superior to Fourier empirical modes.

\section{Mathematical model and the exact solution}\label{model}

The experimental data are provided by the simulation of the nonlinear
viscous Burgers equation model:
\begin{equation}\label{burgers}
\left\{ \begin{array}{l}
\frac{\partial }{{\partial t}}u\left( {x,t} \right) + \frac{\partial }{{\partial x}}\left( {\frac{{u{{\left( {x,t} \right)}^2}}}{2}} \right) = \nu \frac{{{\partial ^2}}}{{\partial {x^2}}}u\left( {x,t} \right),\quad t > 0,\quad \nu  > 0,\\
u\left( {x,0} \right) = {u_0}\left( x \right),\quad x \in \mathbb{R},
\end{array} \right.
\end{equation}
where $u\left( {x,t} \right)$ is the unknown function of time $t$, $\nu $ is the viscosity parameter.
The initial condition of the following form is considered:
\begin{equation}\label{test1}
{u_0}\left( x \right) =  - \sin \left( {\pi x} \right),
\end{equation}
The homogeneous Dirichlet boundary conditions of the form
\begin{equation}\label{boundcond}
u\left( {0,t} \right) = u\left( {L,t} \right) = 0
\end{equation}
are also applied to the model.

The nonlinear evolution governed by the Burgers equation is obtained with the help of the Cole–Hopf transformation defined by:
\begin{equation}\label{cole}
u =  - 2\nu \frac{1}{\varphi }\frac{{\partial \varphi }}{{\partial x}}.
\end{equation}

Through an analytical handling it is found that:
\begin{equation}\label{e1}
\frac{{\partial u}}{{\partial t}} = \frac{{2\nu }}{{{\varphi ^2}}}\left( {\frac{{\partial \varphi }}{{\partial t}}\frac{{\partial \varphi }}{{\partial x}} - \varphi \frac{{{\partial ^2}\varphi }}{{\partial x\partial t}}} \right),\quad u\frac{{\partial u}}{{\partial x}} = \frac{{4{\nu ^2}}}{{{\varphi ^3}}}\frac{{\partial \varphi }}{{\partial x}}\left( {\varphi \frac{{{\partial ^2}\varphi }}{{\partial {x^2}}} - \frac{{\partial \varphi }}{{\partial x}}\frac{{\partial \varphi }}{{\partial x}}} \right),
\end{equation}
\begin{equation}\label{e2}
\nu \frac{{{\partial ^2}u}}{{\partial {x^2}}} =  - \frac{{2{\nu ^2}}}{{{\varphi ^3}}}\left( {2{{\left( {\frac{{\partial \varphi }}{{\partial x}}} \right)}^3} - 3\varphi \frac{{{\partial ^2}\varphi }}{{\partial {x^2}}}\frac{{\partial \varphi }}{{\partial x}} + {\varphi ^2}\frac{{{\partial ^3}\varphi }}{{\partial {x^3}}}} \right).
\end{equation}

Substituting these expressions into (\ref{burgers}) it follows that
\begin{equation}\label{e3}
\frac{{\partial \varphi }}{{\partial x}}\left( {\frac{{\partial \varphi }}{{\partial t}} - \nu \frac{{{\partial ^2}\varphi }}{{\partial {x^2}}}} \right) = \varphi \left( {\frac{{{\partial ^2}\varphi }}{{\partial x\partial t}} - \nu \frac{{{\partial ^3}\varphi }}{{\partial {x^3}}}} \right) = \varphi \frac{\partial }{{\partial x}}\left( {\frac{{\partial \varphi }}{{\partial t}} - \nu \frac{{{\partial ^2}\varphi }}{{\partial {x^2}}}} \right).
\end{equation}

Relation (\ref{e3}) indicates that if $\varphi $ solves the heat equation, then $u\left( {x,t} \right)$ given by the Cole-Hopf transformation (\ref{cole}) solves the viscid Burgers equation (\ref{burgers}). Thus the viscid Burgers equation (\ref{burgers}) is reduced to the following one:
\begin{equation}\label{heat}
\left\{ \begin{array}{l}
\frac{{\partial \varphi }}{{\partial t}} - \nu \frac{{{\partial ^2}\varphi }}{{\partial {x^2}}} = 0,\quad x \in R,t > 0,\nu  > 0,\\
\varphi \left( {x,0} \right) = {\varphi _0}\left( x \right) = {e^{ - \int_0^x {\frac{{{u_0}(\xi )}}{{2\nu }}d\xi } }},x \in \mathbb{R}.
\end{array} \right.
\end{equation}

Taking the Fourier transform with respect to $x$ for both heat equation and the initial condition (\ref{heat}), the analytic solution is obtained in the following form:
\begin{equation}\label{solphi}
\varphi \left( {x,t} \right) = \frac{1}{{2\sqrt {\pi \nu t} }}\int\limits_{ - \infty }^\infty  {{\varphi _0}\left( \xi  \right)} \,{e^{ - \frac{{{{(x - \xi )}^2}}}{{4\nu t}}}}d\xi .
\end{equation}

From the Cole-Hopf transformation (\ref{cole}) we obtain the analytic solution to the problem (\ref{burgers}) in the following form:
\begin{equation}\label{exactsolu}
u\left( {x,t} \right) = \frac{{\int_{ - \infty }^\infty  {\frac{{x - \xi }}{t}{\varphi _0}\left( \xi  \right){e^{ - \frac{{{{\left( {x - \xi } \right)}^2}}}{{4\nu t}}}}d\xi } }}{{\int_{ - \infty }^\infty  {{\varphi _0}\left( \xi  \right){e^{ - \frac{{{{\left( {x - \xi } \right)}^2}}}{{4\nu t}}}}d\xi } }}.
\end{equation}

The exact solution (\ref{exactsolu}) with the initial condition (\ref{test1}) is computed using the Gauss-Hermite Quadrature \cite{Brass2011}. Gauss–Hermite quadrature approximates the value of integrals of the following kind:
\begin{equation}\label{gh}
\int\limits_{ - \infty }^\infty  {f\left( z \right)} \,{e^{ - {z^2}}}dz \approx \sum\limits_{i = 1}^n {{w_i}f\left( {{x_i}} \right)} ,
\end{equation}
where $n$ represents the number of sample points used, ${x_i}$ are the roots of the Hermite polynomial ${H_n}\left( x \right)$ and the associated weights ${w_i}$ are given by
\begin{equation}\label{weights}
{w_i} = \frac{{{2^{n - 1}}n!\sqrt \pi  }}{{{n^2}{{\left( {{H_{n - 1}}\left( {{x_i}} \right)} \right)}^2}}},\quad i = 1,...,n.
\end{equation}

The initial condition (\ref{test1}) leads to
\begin{equation}\label{ex1phi0}
{\varphi _0}\left( x \right) = {e^{ - \frac{1}{{2\nu }}\int_0^x {{u_0}(\xi )d\xi } }} = {e^{ - \frac{1}{{2\nu }}\int_0^x { - \sin \left( {\pi \xi } \right)d\xi } }} = {e^{\frac{1}{{2\nu \pi }}}} \cdot {e^{ - \frac{{\cos \left( {\pi x} \right)}}{{2\nu \pi }}}}.
\end{equation}

The exact solution (\ref{exactsolu}) is written in the following form:
\begin{equation}\label{sol1}
u\left( {x,t} \right) = \frac{{\int_{ - \infty }^\infty  {\frac{{x - \xi }}{t}{e^{ - \frac{{\cos \left( {\pi \xi } \right)}}{{2\nu \pi }}}} \cdot {e^{ - {{\left( {\frac{{x - \xi }}{{\sqrt {4\nu t} }}} \right)}^2}}}d\xi } }}{{\int_{ - \infty }^\infty  {{e^{ - \frac{{\cos \left( {\pi \xi } \right)}}{{2\nu \pi }}}} \cdot {e^{ - {{\left( {\frac{{x - \xi }}{{\sqrt {4\nu t} }}} \right)}^2}}}d\xi } }}.
\end{equation}

Introducing the variable change
\begin{equation}\label{varch}
z = \frac{{x - \xi }}{{\sqrt {4\nu t} }},
\end{equation}
the exact solution to the Burgers equation model (\ref{burgers}) with the initial condition (\ref{test1}) is
\begin{equation}\label{exsol1}
u\left( {x,t} \right) = \frac{{\int_{ - \infty }^\infty  {4\nu z\;{e^{ - \frac{1}{{2\nu \pi }}\cos \left[ {\pi \left( {x - z\sqrt {4\nu t} } \right)} \right]}}{e^{ - {z^2}}}dz} }}{{\int_{ - \infty }^\infty  {\sqrt {4\nu t} \;{e^{ - \frac{1}{{2\nu \pi }}\cos \left[ {\pi \left( {x - z\sqrt {4\nu t} } \right)} \right]}}{e^{ - {z^2}}}dz} }}.
\end{equation}

The computational domain considered is $\left[ {0,L} \right]$, where $L=2$, the computational time is $\left[ {0,T} \right]$, where $T=3$, viscosity parameter in the Burgers equation model is  $\nu  = {10^{ - 2}}$. The computational domain is uniformly discretized by using $N=100$ grid points, which yields a mesh-size $\Delta x = 0.02$.

Eq. (\ref{test1}), representing the initial condition, yields a sinusoidal pulse with an abrupt change of slope at the extremities of the domain. Figure \ref{exact} illustrates the exact solution computed with the technique of the Gauss-Hermite quadrature with $n=100$ nodes.
\begin{figure}
\centering
\includegraphics[scale=1.3]{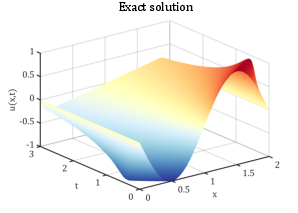}
\caption{The exact solution of the viscid Burgers equation model (\ref{burgers}), with initial condition (\ref{test1})\label{exact}}
\end{figure}

\section{The Twin Data Model. Qualitative analysis}

As a high fidelity substitute model for the precise solution of the Burgers equation model under investigation here, the twin data model (\ref{dtm}) is developed utilizing the randomized orthogonal decomposition (ROD).

The training data comprises of
${N_t}=300$ total number of snapshots taken in time at regularly spaced time intervals $\Delta t = 0.01$,
${N_x}=101$ number of spatial measurements per time snapshot.

The optimal dimension $N_{DTM}=10$ of the leading shape modes space is determined as the solution to the multiobjective optimization problem with nonlinear constraints  (\ref{optimprob}). ROD algorithm solves this problem and finds Pareto front of the two fitness functions using a genetic algorithm. Figure \ref{Pareto} illustrates the objectives of the optimization problem (\ref{optimprob}) and the Pareto front solution, for the considered test case.
\begin{figure}
\centering
\includegraphics[scale=1.3]{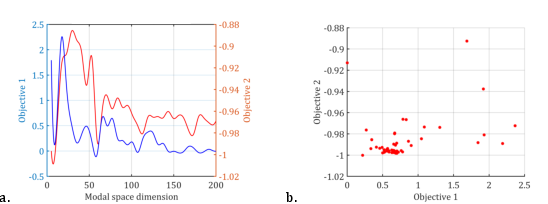}
\caption{The objectives of the optimization problem (\ref{optimprob}) and the Pareto front solution obtained by genetic algorithm \label{Pareto} }
\end{figure}

Figure \ref{dtm1} illustrates the twin data model produced with ROD. The figure presents also the leading shape modes and the corresponding modal amplitudes that contribute to the assembly of the model.
\begin{figure}
\centering
\includegraphics[scale=1.2]{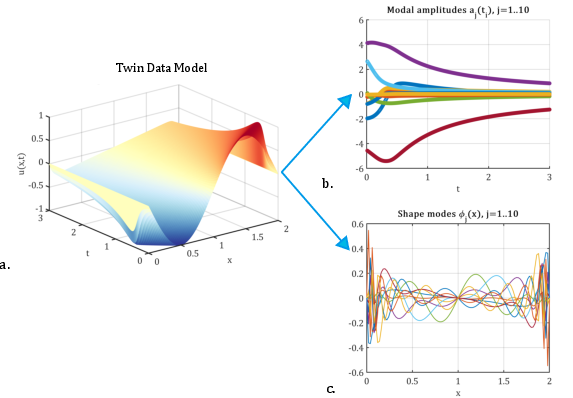}
\caption{a.The twin data model as the solution of ROD algorithm; b.The modal amplitudes; c.The corresponding leading shape modes \label{dtm1} }
\end{figure}

We introduce ${\left\langle  \cdot  \right\rangle _T}$ as a time average operator over $\left[ {{t_1},T} \right]$ corresponding to the arithmetic time-average of equally spaced elements of the interval $\left[ {{t_1},T} \right]$:
\begin{equation}\label{Tnorm}
{\left\langle {f\left( t \right)} \right\rangle _T} = \frac{1}{{{N_t}}}\sum\limits_{i = 1}^{{N_t}} {f\left( {{t_i}} \right)} ,\quad {t_i} \in \left\{ {{t_1},{t_2},...,{t_{{N_t}}} = T} \right\}.
\end{equation}

The absolute error between the exact solution and the twin data model is defined by relation:
\begin{equation}\label{eror}
Erro{r^{DTM}} = {\left\langle {{{\left\| {u\left( {x,t} \right) - {u^{DTM}}\left( {x,t} \right)} \right\|}_2}} \right\rangle _T},\quad t \in \left[ {{t_1},T} \right].
\end{equation}

The correlation coefficient  is used to validate the quality of the
 twin data model over the exact solution and has the following expression:
\begin{equation}\label{corr}
Cor{r^{DTM}} = {\left\langle {\frac{{{{\left\| {u\left( {x,t} \right){u^{DTM}}\left( {x,t} \right)} \right\|}_2}^2}}{{{{\left\| {u{{\left( {x,t} \right)}^2}} \right\|}_2}{{\left\| {{u^{DTM}}{{\left( {x,t} \right)}^2}} \right\|}_2}}}} \right\rangle _T},\quad t \in \left[ {{t_1},T} \right].
\end{equation}

In order to perform a qualitative analysis of the twin data model, Table \ref{qual1} presents the space dimension of the leading shape modes computed by the ROD, the absolute error given by Eq.(\ref{eror}) and the correlation coefficient given by Eq.(\ref{corr}) between the exact solution and the twin data model. Analysing the results presented in Table \ref{qual1}, it is obvious that the ROD algorithm creates a model that is perfectly correlated with the original data (i.e. $Cor{r^{DTM}}=1$), having the absolute error of order $\mathcal{O}({10^{-7}}) $.

Table \ref{qual2} presents the mean squared sum of the projection norms and highlights that the ROD shape modes maximize their projection compared to the Fourier modes. It was demonstrated that the leading modes computed by Randomized Orthogonal Decomposition are qualitatively superior to the Fourier modes.
\begin{center}
\begin{table}[h]
\caption{Qualitative analysis of the twin data model\label{qual1}}
\noindent \fontsize{7}{8}\selectfont
\centering
\begin{tabular}{ | p{3.3cm} | p{3.3cm} | p{3.3cm} |  }
\hline
\multirow{1}{*}{DTM complexity} & \multirow{1}{*}{DTM Absolute error} & \multirow{1}{*}{DTM Correlation} \\
\hline
$N_{DTM}=10$ & ${\rm{8}}{\rm{.7264}} \times {10^{ - 7}}$ & $1.0000$ \\
\hline
\end{tabular}
\end{table}
\end{center}
\begin{center}
\begin{table}
\caption{Qualitative analysis of the shape modes projection\label{qual2}}
\noindent \fontsize{7}{8}\selectfont
\centering
\begin{tabular}{ | p{5.0cm} | p{5.0cm} |   }
\hline
\multirow{1}{*}{ ROD Shape Modes Projection Norm  } & \multirow{1}{*}{ Fourier Modes Projection Norm   } \\
\hline
$\left\| {{P_{{V_0}}}\Phi } \right\|_{{L^2}\left( D \right)}^2=56.6294 $			& $\left\| {{P_{{V_0}}}\Psi } \right\|_{{L^2}\left( D \right)}^2=3.5029$   \\
\hline
\end{tabular}
\end{table}
\end{center}

\section{Conclusions}

A topic of significant interest to data scientists has been the focus of the current study.
The technique of Randomized Orthogonal Decomposition (ROD) was presented in this study in order to produce twin data models of lower complexity that accurately reflect the dynamics of fluid flows. Details of the mathematical framework were provided.

It has been demonstrated that when a multiobjective optimization problem is formulated, randomized orthogonal decomposition outperforms Fourier techniques and reduces the projection error.
It was established that ROD-generated leading shape modes are qualitatively superior than Fourier empirical modes in the sense that ROD modes maximize their projection on the data space.
The twin data model produced using the current method has an exact correlation to the original data.
Traditional approaches that are substantially more computationally expensive and do not necessarily have high precision, such as adjoint model reduction, POD-DEIM, or Galerkin projection methods, can be effectively avoided by Randomized Orthogonal Decomposition.
The suggested approach will be tested on two-dimensional datasets with applications from other domains in a subsequent study.

\end{document}